\documentclass[11 pt,leqno]{amsart}

\usepackage{amsmath}
\usepackage{amssymb}
\usepackage{amsfonts}
\usepackage{amsthm}
\usepackage{hyperref}
\usepackage{verbatim}

\usepackage{bbm}
\usepackage[mathscr]{eucal}

\usepackage{pgfplots}
\pgfplotsset{compat=1.16}
\usetikzlibrary{arrows.meta}
\usetikzlibrary{intersections}

\usepackage{tikz}
\usetikzlibrary{intersections}
\usetikzlibrary{calc}
\usepackage{bbm}
\usepackage{color}\usepackage{enumitem}
\usepackage[mathscr]{eucal}

\usepackage{mathrsfs}
\usepackage{tcolorbox}
\usepackage{color}

\makeatletter
\def\subsection{\@startsection{subsection}{2}%
  \z@{.5\linespacing\@plus.7\linespacing}{-.5em}%
  {\normalfont\itshape}}
\makeatother

\makeatletter
\def\author@andify{%
  \nxandlist{\unskip\penalty-1\;\;\space\ignorespaces}%
           {\unskip\penalty-1\;\;\space\ignorespaces}%
           {\unskip\penalty-2\;\;\space\ignorespaces}%
}

\makeatother

\theoremstyle{plain}
\newtheorem{thm}{Theorem}[section]
\newtheorem{lemma}[thm]{Lemma}

\newtheorem{proposition}[thm]{Proposition}

\theoremstyle{remark}

\newtheorem*{remarka}{Remark}

\numberwithin{equation}{section}

\newcommand\Bellc{{B_{\ell+1}^\complement}}

\newcommand{\inn}[2]{\langle#1,#2\rangle}

\begin{document}



\title[Endpoint power weight inequalities]{Endpoint power weight inequalities for spherical maximal functions}

\author[M. Fraccaroli, J. Lee, N. Lei, J. Roos, A. Seeger]{Marco Fraccaroli \ Juyoung Lee \ Nikou Lei \\ Joris Roos \ Andreas Seeger}
\address{M.F., Department of Mathematics and Statistics, University of Massachusetts Lowell, Lowell, MA, USA}
\email{marco\_fraccaroli@uml.edu}

\address{J.L., Department of Mathematics, University of British Columbia, Vancouver, BC, Canada}
\email{juyounglee@math.ubc.edu}

\address{N.L., Department of Mathematics, University of British Columbia, Vancouver, BC, Canada}
\email{leinikou@math.ubc.ca}
\address{J.R., Department of Mathematics and Statistics, University of Massachusetts Lowell, Lowell, MA, USA}
\email{joris\_roos@uml.edu}

\address{A.S., Department of Mathematics, University of Wisconsin--Madison, Madison, WI, USA}
\email{aseeger@wisc.edu}

\begin{abstract} We prove sharp $L^p(|x|^\alpha)$ estimate for spherical maximal functions. We resolve the remaining open problem for the full maximal operator for $\alpha=1-d$ and $p<2$ and prove a more general result
involving a local version of Assouad dimension of the underlying dilation set. We also make a case study with endpoint bounds for a model case involving convex sequences in $[1,2]$.

\end{abstract}

\maketitle

\newcommand{\todo}[1]{{\color{blue}ToDo: #1}}

\section{Introduction}

Let $d\ge 2$ and $\sigma$ the normalized surface measure on $\{x\in {{\mathbb{R}}}^d:|x|=1\}$. Define the dilate $\sigma_t$ by $\inn{\sigma_t}{f}=\inn{\sigma}{f(t\cdot )}$.
For a $f\in L^1_{{\text{\rm loc}}}({{\mathbb{R}}}^d)$, let
\[A_tf(x) =f*\sigma_t(x)=\int f(x+t{\omega}) d\sigma({\omega}).\]
Given ${{\mathcal{E}}}\subset (0,\infty)$ define
the maximal operator ${{\mathcal{M}}}_{{\mathcal{E}}} $ by
\[{{\mathcal{M}}}_{{\mathcal{E}}} f(x)= \sup_{t\in {{\mathcal{E}}}} |A_t f(x)|,\]
which is well defined as a measurable function at least when $f$ is continuous.
Let $w_\alpha(x)=|x|^\alpha$.
In this paper we continue the investigation of boundedness of the operator ${{\mathcal{M}}}_{{\mathcal{E}}}$ on weighted $L^p(w_\alpha)$ spaces in
\cite{DuoandikoetxeaVega, DuoandikoetxeaSeijo, LeeJ-weighted, FraccaroliRoosSeeger}. We note that when $p>1$, $w_\alpha$ is an $A_p$ weight (\cite{GarciaCuerva-Rubio}) for $-d<\alpha<d(p-1)$, but more restrictive conditions are needed for boundedness on the spherical maximal operators.

For the lacunary maximal operator ${{\mathcal{M}}}_{{{\mathcal{E}}}^{\mathrm{lac}}}$
with
${{\mathcal{E}}}^{\mathrm{lac}}=\{2^k:k\in {{\mathbb{Z}}}\}$ it was proved by Duoandikoetxea and Vega \cite{DuoandikoetxeaVega} that ${{\mathcal{M}}}_{{{\mathcal{E}}}^{\mathrm{lac}}}$ is bounded on $L^p(w_\alpha)$ if and only if $1-d\le \alpha<(d-1)(p-1)$, $p>1$.

For the full maximal operator ${{\mathcal{M}}}:={{\mathcal{M}}}_{(0,\infty)}$ the $L^p$ inequalities in the unweighted case are due to Stein \cite{SteinPNAS1976} and Bourgain \cite{BourgainJdA1986}, and the $L^p(w_\alpha)$ boundedness was almost fully covered in \cite{DuoandikoetxeaVega} where it was shown that boundedness holds when $p>\frac{d}{d-1}$ and $1-d<\alpha<(d-1)p-d$; moreover boundedness fails for all $\alpha$ if $p\le\frac{d}{d-1}$, and for $p>\frac{d}{d-1}$ when $\alpha \notin [1-d, (d-1)(p-1)-1) $. This also applies for the local maximal function $\sup_{t\in[1,2]}|A_t f|$. In \cite{DuoandikoetxeaVega} the case $p>\frac{d}{d-1}$, $\alpha=1-d$ was left open. The second named author proved more recently
\cite{LeeJ-weighted} that ${{\mathcal{M}}}$ is bounded on $L^p(w_{1-d})$ in the cases
$d=2$, $p>2$, and $d\ge 3$, $p\ge 2$; in particular this gives a complete answer in two dimensions. We complete the story for the full maximal operator by adding the relevant result for $p<2$; moreover we will simplify the proof in \cite{LeeJ-weighted} for $p\ge 2$.

\begin{thm}\label{thm:newfull}
For $d\ge 3$, ${{\mathcal{M}}}$ is bounded on $L^p(w_{1-d})$ for all $p>\frac{d}{d-1}$.
\end{thm}

For general ${{\mathcal{E}}}\subset (0,\infty)$ the behavior of $L^p$ and $L^p(w_\alpha)$ boundedness depends on various notions of fractal dimensions.

To describe these first let, for $A\subset (0,\infty)$, $N(A,\delta)$ be the minimal number of intervals of length $\delta$ needed to cover $A$ (i.e. the $\delta$ metric entropy of $A$).
For $R>0$ let \[{{\mathcal{E}}}_R=R^{-1} ({{\mathcal{E}}}\cap[R,2R)). \]

Define $\beta=\beta({{\mathcal{E}}})$ as
\begin{equation}
\beta= \inf \big\{\eta: \sup_R\sup_{0<\delta\le 1} \delta^\eta N({{\mathcal{E}}}_R,\delta) <\infty \} .
\end{equation}
For $E\subset [1,2]$, $\beta(E)$ is simply the upper Minkowski-dimension of $E$. By \cite{SeegerWaingerWright1995} $M_{{\mathcal{E}}}$ is bounded on $L^p({{\mathbb{R}}}^d)$ for $p>1+\frac{\beta}{d-1}$, and unbounded for $p<1+\frac{\beta}{d-1}$, see also \cite{STW-jussieu}
for some endpoint results and \cite{SeegerWaingerWright1997,RoosSeeger-problems} for open problems suggested by the behavior on radial functions.

Let $\gamma=\gamma({{\mathcal{E}}})$ be defined as
\begin{equation} \gamma= \inf \Big\{\eta: \sup_R \sup_{\substack{0<\delta \le |I| \le 1\\ \text{$I$ subinterval of $[1,2]$}}} \Big(\frac\delta{|I|}\Big) ^{\eta}N({{\mathcal{E}}}_R\cap I,\delta) <\infty\Big\} .\end{equation}
For $E\subset [1,2]$, $\gamma(E)$ is the Assouad dimension of $E$.
For the full dilation set we have $\gamma=1$, so the following theorem on $L^p(w_\alpha)$ boundedness contains Theorem \ref{thm:newfull} as a special case.
\begin{thm}\label{thm:newfullga}
Let $d\ge 2$, and let $\gamma=\gamma({{\mathcal{E}}})$. Then $M_{{\mathcal{E}}}$ is bounded on $L^p(w_{1-d})$ for all $p>1+ \frac{\gamma}{d-1}$.
\end{thm}

For the range $1+\frac\beta{d-1}\le p\le 1+\frac{\gamma}{d-1}$ we need to put additional restrictions on $\alpha$ (relying on other fractal quantities) to get $L^p$ boundedness.

For $\theta\in [0,1)$ let
$\gamma_\theta=\gamma_\theta({{\mathcal{E}}})$ be defined as
\begin{equation} \gamma_\theta= \inf \Big\{\eta: \sup_R \sup_{\substack{0<\delta \le 1 \\ \text{$I$ subinterval of $[1,2]$
}\\\delta^\theta \le |I|\le 1}} \Big(\frac\delta{|I|} \Big) ^{\eta}N({{\mathcal{E}}}_R\cap I,\delta) <\infty\Big\} .\end{equation}
For $E\subset [1,2]$ the function $\theta \to \gamma_\theta(E)$ is the upper Assouad spectrum of dimensions, defined in \cite{FraserHareHareTroscheitYu}.

Let \begin{equation}
{{\mathfrak{N}}}_{{\mathcal{E}}}^\sharp(b,\delta) =\sup_{R>0} \,\,
\sup_{\substack{\text{$I\subset [1,2]$} \\
\delta<|I|\le 1
}}|I|^{-b} N({{\mathcal{E}}}_R\cap I,\delta)
\end{equation}
and \begin{equation}
\nu^\sharp _{{\mathcal{E}}} (b) = \inf \{ \eta: \sup_{0<\delta\le 1} \delta^\eta {{\mathfrak{N}}}_{{\mathcal{E}}}^\sharp(b,\delta) <\infty\} .
\end{equation}
For $E\subset [1,2]$ the function $b\mapsto \nu^\sharp_{{\mathcal{E}}}(b) $ is the
{\it Legendre--Assouad function} of $E$ introduced in \cite{BeltranRoosSeeger-radial}; it can be identified with the Legendre transform of the function $\theta\mapsto (\theta-1)\gamma_{\theta}(E)$, see \cite{BeltranRoosRutarSeeger}. By a use of a dilation invariant metric it can be interpreted as a Legendre--Assouad type function even for general sets ${{\mathcal{E}}}$ (see \cite{FraccaroliRoosSeeger}).

We recall some properties of $\nu_{{\mathcal{E}}}^\sharp$ ({\emph{cf.\ }} \cite{BeltranRoosSeeger-radial, BeltranRoosRutarSeeger}).
\begin{itemize}
\item $\nu^\sharp_{{\mathcal{E}}}$ is continuous, nondecreasing and convex on ${{\mathbb{R}}}$.

\item $\nu^\sharp_{{\mathcal{E}}} (b)=b$ for $b\ge \gamma$.

\item There exists $b_*\in [0, \gamma]$ such that $\nu^\sharp_{{\mathcal{E}}} (b)=\beta$ for $b\le b_*$ and $\nu^\sharp_{{\mathcal{E}}}$ is strictly increasing on $[b_*,\infty)$.
\end{itemize}
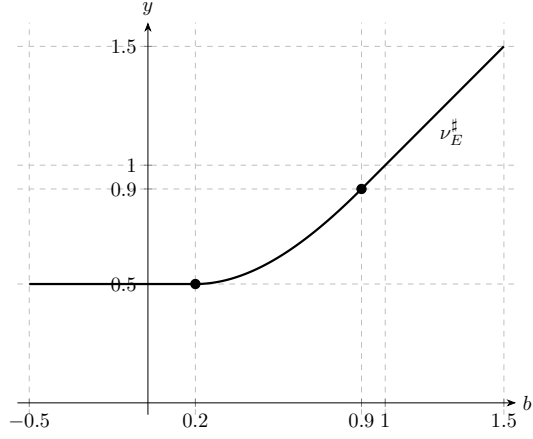
\begin{figure}
\begin{tikzpicture}[scale=.7]
\begin{axis}[
width=11cm,
height=11cm,
axis lines=middle,
axis line style={-Stealth},
axis equal image,
xmin=-0.55, xmax=1.55,
ymin=-0.05, ymax=1.60,
xtick={-0.5,0,0.2,0.9,1,1.5},
ytick={0,0.5,0.9,1,1.5},
grid=major,
grid style={dashed,gray!50},
xlabel={$b$},
ylabel={$y$},
xlabel style={at={(ticklabel* cs:1)}, anchor=west},
ylabel style={at={(ticklabel* cs:1)}, anchor=south},
clip=false
]

\addplot[
very thick,
domain=-0.5:0.2,
samples=2
]{0.5};

\addplot[
very thick,
domain=0.2:0.9,
samples=150
]{0.5 + (50/49)*(x-0.2)^2 - (100/343)*(x-0.2)^3};

\addplot[
very thick,
domain=0.9:1.5,
samples=2
]{x};

\addplot[
only marks,
mark=*,
mark size=2.5pt,
] coordinates {(0.2,0.5) (0.9,0.9)};

Label\node[black] at (axis cs:1.28,1.14) {$\nu_E^\sharp$};

\end{axis}
\end{tikzpicture}
\caption{Legendre--Assouad function for a set $E\subset[1,2]$ with $\beta=1/2$, $b_*=1/5$ and $\gamma=9/10$.}
\end{figure}

The main result in \cite{FraccaroliRoosSeeger} states that we can characterize the closure of the `weighted type set'
\begin{equation}\label{eq:fW}
{{\mathfrak{W}}}_{{\mathcal{E}}}=\{ (\tfrac 1p, \tfrac \alpha p)\in [0,1]\times \mathbb{R} \;: \; M_{{\mathcal{E}}} \text{ is bounded on } L^p(w_\alpha) \}.
\end{equation}
To describe this let $(\nu^\sharp)^{-1}:[\beta,\infty)\to [b_*,\infty)$ be the inverse of the restriction of $\nu^\sharp$ to $[b_*,\infty)$ and let for $p\ge p_\beta
=1+\frac{\beta}{d-1}$
\begin{equation} \begin{aligned}
L(p) &=(d-1)(p-2)-
(\nu^\sharp)^{-1} \big((d-1)(p-1) \big),
\\U(p)&= (d-1)(p-1)-\beta.
\end{aligned}
\end{equation}
Then for $d\ge 2$ and nonempty ${{\mathcal{E}}}$
\begin{equation}\label{eq:typesetexplicit} \overline{{{\mathfrak{W}}}_{{\mathcal{E}}}} = \{ \big(\tfrac 1p,\tfrac\alpha p\big )\in [0,1]\times \mathbb{R}\,:\, p\ge p_\beta,\,\, L(p)\le \alpha\le U(p)\}.
\end{equation}
By the property $\nu^\sharp_{{\mathcal{E}}}(b)=b$ for $b\ge \gamma$ we see that $L(p)=1-d$ for $p>1+\frac{\gamma}{d-1}$. This is consistent with Theorem \ref{thm:newfullga}

\subsection*{Two necessary conditions} Let $E\subset [1,2]$. Testing $M_E$ on the characteristic function of
$\{y: \delta/2\le |y|\le \delta\}$, for small $\delta$, leads to the lower bound
\begin{equation}\label{lowerboundU(p)}
\|M_E\|_{L^p(w_\alpha) \to L^p(w_\alpha)}^p {\gtrsim} \sup_{0<\delta<1} N(E,\delta)\delta^{(d-1)(p-1) -\alpha }
\end{equation}
and thus to the necessary condition $\alpha\le U(p)$ for $L^p(w_\alpha)$-boundedness.
Another necessary condition involves the quantity
\begin{equation} \label{eq:Lambdadef}
{\Lambda}^E_{k,j} (p,\alpha)= 2^{-j(d-1)(p-1)} \sup_{\substack{I\subset [1,2] \\ |I|\approx 2^{-k} } } N(E\cap I,2^{-j} ) 2^{-k((d-1)( 2-p)+\alpha) }
\end{equation}
Indeed we have from \cite{FraccaroliRoosSeeger}
\begin{equation} \label{eq:lowerbdL}
\|M_E\|_{L^p(w_\alpha) \to L^p(w_\alpha)}^p{\gtrsim} \sup_{j\ge k\ge 0} \Lambda_{k,j}^E(p,\alpha)
\end{equation}
which leads to the necessary condition $\alpha\ge L(p)$.

One might ask whether for $p\ge 1+\frac\beta{d-1}$ the condition \[ \sup_{0<\delta<1} N(E,\delta)\delta^{(d-1)(p-1) -\alpha}<\infty \] is also sufficient for $L^p(w_\alpha)$-boundedness when $\alpha=U(p)$ but this is open (in particular for $\alpha=0$, $p=1+\frac{\beta}{d-1}$).
Regarding the lower boundary of ${{\mathfrak{W}}}_E$ one may also ask the question when
the condition
$\sup_{j\ge k\ge 0} \Lambda_{k,j}^E(p,\alpha)<\infty $ is enough for $L^p(w_\alpha)$ boundedness when $\alpha= L(p)$. This is is mostly unknown; indeed we shall see that there are instances when this condition holds with $\alpha=L(p)$ but $L^p(w_\alpha)$ boundedness fails (see Theorem \ref{thm:sequences} below for the case $\beta<1$, $p=\frac{d}{d-1}$ and $\alpha=1-d$).

In view of the incomplete knowledge of endpoint bounds for general $E$ we make a case study of a model case
for which the results in the unweighted case are complete. Namely we consider
for $\mu>0$ the sequence set
\begin{equation} \label{eq:defEmu}E^\mu=\{ 1+ n^{-\mu}: n\in {{\mathbb{N}}}\}, \end{equation}
for which $N(E,\delta)\approx \delta^{-1/(\mu+1)}$ and
$\sup_{j\ge k\ge 0} \Lambda_{k,j}^E(p,\alpha)<\infty $ holds.
Thus $\beta(E^\mu)= \frac{1}{\mu+1}$, moreover $\gamma(E^\mu)=1$, and
\[\nu_{E^\mu}^\sharp (b)= \frac{\mu}{\mu+1} b+ \frac 1{\mu+1} \qquad \text{ for $0\le b\le 1$.}
\] The corresponding functions $U$ and $L$ are given by
\[ U(p)= (d-1)(p-1)-\frac 1{\mu+1}\] and
\begin{equation}
L(p) =\begin{cases} 1-d+ \frac{d-(d-1)p}{\mu} &\text{ for } 1+\frac{1}{(\mu+1)(d-1) }\le p\le \frac{d}{d-1}
\\
1-d &\text{ for } p\ge \frac d{d-1}
\end{cases}
\end{equation}
For $p_\beta= p(\mu)=1+\frac 1{(\mu+1)(d-1) }$ we have $L(p(\mu))= 1-d+\beta= 1-d+\frac{1}{\mu+1}$.

\begin{thm}\label{thm:sequences}
Let $d\ge 2$ and $p(\mu)= 1+ \frac{1}{(\mu+1)(d-1)} $.
For $p(\mu) <p<\infty$ the operator ${{\mathcal{M}}}_{E^\mu} $ is bounded on $L^{p} (w_\alpha)$ if and only if
\begin{align*}
& \frac{d}{d-1}<p<\infty, && 1-d\le \alpha\le (d-1)(p-1)-\frac{1}{\mu+1}
\\
\text{ or }\, &p= p_d:=\frac{d}{d-1}, &&1-d< \alpha\le (d-1)(p_d-1)-\frac{1}{\mu+1}
\\
\text{ or }\,
&p(\mu)<p<\frac{d}{d-1}, && 1-d +\frac{d-(d-1)p}{\mu}\le \alpha \le (d-1)(p-1)-\frac{1}{\mu+1}.
\end{align*}

\end{thm}

For $p=p_d$ note that we have boundedness for $(d-1)(p_d-1)-\frac{1}{\mu+1}= \frac{\mu}{\mu+1}$, and unboundedness for $\alpha=1-d$; the latter corresponds to the corner on the line $\frac{\alpha}{(d-1)p}=-\frac{d-1}{d}$ in Figure \ref{fig2}.
For $p<p(\mu)$ the operator $M_{E^\mu}$ is not bounded on any $L^p(w_\alpha)$. For $p=p(\mu)$, $M_{E^\mu}$ is not bounded on $L^p(w_\alpha)$, if $\alpha<1-d+(\mu+1)^{-1}$, by \cite{FraccaroliRoosSeeger}. The case $\alpha=0$ of (ii) was proved in \cite{STW-jussieu}. The behavior of $M_{E^\mu}$ on $L^{p(\mu)} (w_\alpha)$ remains open for $\alpha=1-d+(\mu+1)^{-1}$.
\begin{figure}
\begin{tikzpicture}[scale=.5]
\begin{axis}[
width=11cm,
height=11cm,
axis lines=middle,
axis line style={-Stealth},
axis equal image,
xmin=0, xmax=1.6,
ymin=-1.4, ymax=0.7,
xtick={1,1.5},
xticklabels=\empty,
ytick={-1,-1.3},
yticklabel=\empty,
grid=major,
grid style={dashed,gray!50},
xlabel={$\frac{1}{p}$},
ylabel={$\frac{\alpha}{(d-1)p}$},
xlabel style={at={(ticklabel* cs:1)}, anchor=west},
ylabel style={at={(ticklabel* cs:1)}, anchor=south},
clip=false
]

\addplot[
very thick,
domain=0:0.985,
samples=2
]{-x};

\addplot[
very thick,
domain=0:1.5,
samples=150
]{0.5 -x/3};

\addplot[
very thick,
domain=1.015:1.485,
samples=2
]{(-3*x/5)-(2/5)};

\addplot[
only marks,
mark=*,
mark size=2.5pt,
] coordinates {(1.5,0)};

\addplot[
only marks, very thick,
mark=*,
mark size=2.5pt, black
] coordinates {(1,-1)};

\addplot[
only marks,
mark=o,
mark size=2.5pt,
] coordinates {(1.5,-1.3)};

\addplot[
very thick
] coordinates {
(1.5,0)
(1.5,-1.285)
};

\addplot[
very thick
] coordinates {
(0,0)
(0,0.5)
};

Label\node[black] at (axis cs:1.4,-0.1) {$\frac{1}{p(\mu)}$};
Label\node[black] at (axis cs:0.9,-0.1) {$\frac{d-1}{d}$};
Label\node[black] at (axis cs:-0.05,0.5) {$1$};
Label\node[black] at (axis cs:-0.13,-1) {$-\frac{d-1}{d}$};
Label\node[black] at (axis cs:-0.17,-1.3) {$\frac{L(p(\mu))}{(d-1)p(\mu)}$};

\draw[
fill=black,
fill opacity=0.1,
draw=none
]
(axis cs:0,0)
-- (axis cs:0,0.5)
-- (axis cs:1.5,0)
-- (axis cs:1.5,-1.3)
-- (axis cs:1,-1)
-- cycle;

\end{axis}
\end{tikzpicture}
\caption{Sharp type set ${{\mathfrak{W}}}_{E^\mu}$, vertically scaled by $(d-1)^{-1}$.
}
\label{fig2}\end{figure}

For the global analogue the $L^p(w_\alpha)$ boundedness fails for $\alpha=U(p)$.
\begin{thm}
Let ${{\mathcal{E}}}^\mu=\cup_{k\in {{\mathbb{Z}}}}2^kE_\mu$ and $p(\mu)<p<\infty$. Then $M_{{{\mathcal{E}}}^\mu}$ is bounded on $L^p(w_\alpha)$ if and only if $\alpha <(d-1)(p-1)-(\mu+1)^{-1}$ and $M_{E^\mu}$ is bounded on $L^p(w_\alpha)$.
\end{thm}

\subsection*{Statement on use of AI}
No AI tools were used for the mathematics or the writing of this paper. Prior to this project,
GPT 5.5 had been used to assist with a previous version of Figure 1.

\subsection*{Acknowledgements}
M.F. was supported in part by an AMS-Simons travel grant. J.R. was supported in part by NSF grant DMS-2555784 and a Simons Foundation grant. A.S. was supported in part by NSF grant DMS-2348797.

\noindent {\em Remark.} While we were finishing this manuscript, an independent preprint by Basak, Choudhary, Shen and Spector \cite{BCSS26} appeared on the arXiv with overlap in
the main results, in particular
Theorem \ref{thm:newfull}.

\section{Upper bounds}
In this section we prove Theorem \ref{thm:newfullga}.
In the case $d=2$, $\gamma=1$ the claim follows from Theorem 1.1 (1) of \cite{LeeJ-weighted}. Thus we may exclude this case in the following.
We will consider the case $E\subset [1,2].$
Let $B_k=\{x:|x|\le 2^{-k}\}$ and $B_k^\complement$ its complement.
For $\ell\ge 4$
let ${{\mathscr C}}(\ell)$ be the infimum over all constants $C$ such that
\[ \big\|\sup_{t\in E} |A_t f | \big \|_{L^p(\Bellc,w_{1-d})} \le C \|f\|_{L^p(w_{1-d})}\]
for all $f\in L^p(w_{1-d})$ . We shall first show that ${{\mathscr C}}(\ell)<\infty$ for all $\ell>0$, and subsequently show that
${{\mathscr C}}(\ell)=O(1)$ as $\ell\to\infty$.

For continuous $f$,
$\sup_{\ell \geq 0} {{\mathbbm 1}} _{\Bellc}\sup_{t\in E} |A_t f (x)| = \sup_{t\in E} |A_t f(x)|$. Observe that continuous $L^p(w_\alpha)$ functions are dense in $L^p(w_\alpha)$ if $\alpha>-d$. Using the monotone convergence theorem we see that it suffices to prove ${{\mathscr C}}(\ell)= O(1)$ as $\ell\to \infty$.

Starting out as in \cite{FraccaroliRoosSeeger}
we split
\begin{equation}
\label{eq:lowhigh}f=f_{\mathrm{low}} + {f_{\mathrm{mid}}} + f_{\mathrm{high}},
\end{equation} where $f_{\mathrm{low}} $ lives on $ \{y:|y|<1/4\}$, ${f_{\mathrm{mid}}}$ lives on $\{y: 1/2\le |y|\le 4\}$, and
$f_{\mathrm{high}} $ lives on $ \{y:|y|>4\}$.
The $L^p(w_\alpha)$ estimates
\begin{align*}
\big\|\sup_{t\in E} |A_t f_{\mathrm{low}} | \big \|_{L^p(w_\alpha)} {\lesssim} \|f\|_{L^p(w_\alpha)} \\
\big\|\sup_{t\in E} |A_t f_{\mathrm{high}} | \big \|_{L^p(w_\alpha)} {\lesssim} \|f\|_{L^p(w_\alpha)}
\end{align*}
follow from the unweighted bounds and are valid for $p> 1+\frac\beta{d-1}$ (and nothing needs to be proved since $\beta\le \gamma$). See \cite{FraccaroliRoosSeeger}. By the same argument
\[ \Big( \int_{|x|\ge 2^{-4}} \sup_{t\in E} |A_t {f_{\mathrm{mid}}} (x) |^p w_\alpha(x) dx\Big)^{1/p} {\lesssim} \|f\|_{L^p(w_\alpha)}. \]

In particular we get, for some constant $C_0$ (depending on the unweighted $L^p$ estimates)
\begin{align}\label{eq:lowconst}
\big\|\sup_{t\in E} |A_t f_{\mathrm{low}} | \big \|_{L^p(w_{1-d})} \le C_0 \|f\|_{L^p(w_{1-d})} \\
\label{eq:highconst}
\big\|\sup_{t\in E} |A_t f_{\mathrm{high}} | \big \|_{L^p(w_{1-d})} \le C_0 \|f\|_{L^p(w_{1-d})}
\\
\label{eq:midconst}
\Big( \int_{|x|\ge 2^{-4}} \sup_{t\in E } |A_t {f_{\mathrm{mid}}} (x) |^p w_{1-d} (x) dx\Big)^{1/p} \le C_0 \|f\|_{L^p(w_{1-d})}
\end{align}
with $C_0$ independent of $\ell$.

It is our task to show that
\begin{equation}\label{eq:goal}
\Big( \int_{2^{-\ell-1}\le |x|\le 2^{-4}} \sup_{t\in E } |A_t {f_{\mathrm{mid}}} (x) |^p w_{1-d} (x) dx\Big)^{1/p} {\lesssim} \|f\|_{L^p(w_{1-d})}
\end{equation} with an implicit constant uniform in $\ell$.

We first note that by the unweighted estimates
for $p>1+\frac{\beta}{d-1}$
\begin{align*}
&\Big( \int_{2^{-\ell-1}\le |x|\le 2^{-4}} \sup_{t\in E } |A_t {f_{\mathrm{mid}}} (x) |^p w_{1-d} (x) \,dx\Big)^{1/p} \\&{\lesssim} 2^{\ell(d-1)/p} \big\| \sup_{t\in E } |A_t {f_{\mathrm{mid}}} |\big\|_{p}
{\lesssim} 2^{\ell(d-1)/p} \|{f_{\mathrm{mid}}}\|_p {\lesssim} 2^{\ell(d-1)/p} \|f\|_{L^p(w_{1-d})}.
\end{align*}
This estimate, together with \eqref{eq:lowconst}, \eqref{eq:highconst} and \eqref{eq:midconst} implies that
${{\mathscr C}}(\ell)$ is finite for all $\ell>0$, indeed ${{\mathscr C}}(\ell) {\lesssim} 2^{\ell(d-1)/p}$.

We proceed to show \eqref{eq:goal} which then in turn implies ${{\mathscr C}}(\ell)=O(1)$. We use a spatially localized Littlewood-Paley decomposition, using the Littlewood-Paley lemma with a large number of vanishing moments in \cite[Lemma 2.1]{SeegerTao}, to ensure some reproducing property.\footnote{This is not essential for the argument but it is convenient as it prevents certain error terms from showing up.} In this lemma let $M$ large (we need $M\gg \frac{d-1}{2}$). One constructs radial $C^\infty$ functions $\psi_0, \psi_1, \widetilde \psi_0, \widetilde \psi_1$, all supported in $\{x:|x|\le \frac 1{200}\}$ such that $\widehat\psi_0(\xi)-1=O(|\xi|^M)$, $\widehat{\widetilde \psi_0}(\xi)-1=O(|\xi|^M)$, and $\widehat\psi_1(\xi)=O(|\xi|^M)$, $\widehat{\widetilde \psi_1}(\xi)
=O(|\xi|^M)$. Moreover, with $\psi_k = 2^{(k-1) d}\psi_1 (2^{k-1} \cdot)$,
$\widetilde \psi_k = 2^{(k-1) d}\widetilde \psi_1 (2^{k-1} \cdot)$, we have
\begin{equation}\label{eq:resolotionofid} f= \widetilde\psi_0*\psi_0*f+ \sum_{k=1}^\infty \widetilde \psi_k*\psi_k*f .\end{equation}
Moreover, for $\ell\ge 0$
\begin{equation}\label{eq:telescoping}\sum_{k=0}^\ell \widetilde\psi_k* \psi_k= 2^{\ell d}\Phi_0(2^\ell \cdot)\end{equation}
with $\Phi_0\in C^\infty_c$ and $\int \Phi_0=1$.
We set $L_j f=\psi_j*f$ and $A_t^j f(x) =\sigma_t*\widetilde \psi_j*f$,
so that $$A_t f=\sum_{j=0}^\infty A^j_t L_j f .$$
We shall need some elementary estimates for $A_t^j {f_{\mathrm{mid}}}(x)$ when evaluated for $|x|\approx 2^{-k}$.

\begin{lemma}\label{lem:k>j} For $0\le j<k$ we have
\[ \label{jk-ineq-j<k} \Big(\int_{ |x|\le 2^{-k}} \sup_{t\in [1,2]} |A_t^j g(x)|^p w_{1-d} (x) dx\Big)^{\frac 1p}
\le C 2^{(j-k)/p} \|g\|_p\,\] provided that $g$ is continuous and $1\le p<\infty$.
\end{lemma}
A similar lemma is already in \cite{LeeJ-weighted} and proved for $p=2$ using a more elaborate argument based on spherical harmonics expansions. We give a shorter simple proof here.

\begin{proof}
We show the $L^p\to L^p(w_{1-d})$ inequality for $p=1$ and for $p=\infty$ (for continuous $g$).
For continuous $g$ the function $(x,t)\mapsto A_t g(x)$ is continuous on ${{\mathbb{R}}}^d\times (0,\infty)$.
The convolution kernels $K^j_t$ of $A ^j_t$ satisfy $\sup_{t\in [1,2]} \|K^j_t\|_1 {\lesssim} 1$, hence
\[ \sup_x \sup_{t\in [1,2]} |A^j_t g(x)| =\sup_{t\in [1,2]} \sup_x |K_t^j*g(x)|{\lesssim} \|g\|_\infty.\]

For $p=1$ we use $\|K^j_t\|_\infty=O(2^j)$ to bound
\begin{multline*}
\int_{|x|\le 2^{-k}} \sup_{t\in [1,2]} |A_t^j g(x)||x|^{1-d} dx
{\lesssim}
\int_{|x|\le 2^{-k}} |x|^{1-d} dx \sup_x
\sup_{t\in [1,2]} |A_t^j g(x)|
\\
{\lesssim} 2^{-k}
\sup_{t\in [1,2]} \sup_x|K^j_t* g(x)|
{\lesssim} 2^{j-k} \|g\|_1.
\end{multline*}
For intermediate $p$ the inequality follows by interpolation.
\end{proof}

The following lemma with its proof is taken from \cite{FraccaroliRoosSeeger} (similar arguments were also used in \cite{DuoandikoetxeaSeijo}).
\begin{lemma}\label{lem:k<j}
Let $1\le p\le 2$.
Let $E\subset[1,2]$ and
$B_k=\{x:|x|\le 2^{-k} \}$. Let, for $k\le j$,
\begin{equation} \label{eq:Lambdakjdef}
{\Lambda}^E_{k,j}(p,\alpha) = 2^{-j(d-1)(p-1)}
\sup_{\substack{I\subset [1,2] \\ |I|\approx 2^{-k} } }N(E\cap I,2^{-j} ) 2^{-k((d-1)( 2-p)+\alpha) }
\end{equation}
Then for $k\le j$.
\begin{equation} \label{jk-ineq} \Big(\int_{B_k\setminus B_{k+1}} \sup_{t\in E} |A_t^j g(x)|^p|x|^\alpha dx\Big)^{1/p} \\
\lesssim [{\Lambda}^E_{k,j}(p,\alpha) ]^{1/p} \|g\|_p
\end{equation}
\end{lemma}
\begin{proof}

We let $I_{k,\mu} =[ 2^{-k} (\mu-1) , 2^{-k}( \mu+2)].$ For each $\mu$ with $E\cap I_{k,\mu} \neq \emptyset$, denote by $\Gamma_{k,\mu,j}$ a collection of intervals of length $2^{-j}$ covering $E\cap I_{k,\mu}$ such that $N(E\cap I_{k,\mu} , 2^{-j}) \approx \# \Gamma_{k,\mu,j} $. If $J\in \Gamma_{k,\mu,j}$ we let $t_J$ be the center of $J$.

We first prove that for $1\le p\le 2$, $k<j$,
\begin{multline}\label{eq:mainunweighted}
\Big( \int_{B_k\setminus B_{k+1} } \sup_{t\in E\cap I_{k,\mu}} |A_t^j g |^p dx\Big)^{\frac 1p}\\ {\lesssim} 2^{-j(d-1)(1-\frac 1p)} 2^{-k(d-1)(\frac 2p-1)} N(E\cap I_{k,\mu} , 2^{-j}) ^{\frac 1p} \|g\|_p.
\end{multline}
We use the pointwise estimate
\begin{multline} \label{eq:ptwest} \sup_{t\in E\cap I_{k,\mu}} |A_t^j g(x) |
\\ \le \Big(\sum_{J\in {\Gamma}_{k,\mu,j}} |A_{t_J}^j g(x) |^p \Big)^{\frac 1p}
+ \Big(\sum_{J\in {\Gamma}_{k,\mu,j} }\Big(\int_{|s|\le 2^{-j}} |\tfrac{d}{ds}A_{t_J+s}^j g(x) |ds \Big) ^p \Big)^{\frac 1p}.
\end{multline}
By the integral Minkowski inequality \eqref{eq:ptwest} follows from
\begin{subequations}
\begin{multline}\label{eq:jkmu-1}
\Big( \sum_{J\in {\Gamma}_{k,\mu,j}} \int_{B_k\setminus B_{k+1} } |A_{t_J} ^j g (x)|^p dx\Big)^{\frac 1p}
\\ {\lesssim} 2^{-j(d-1)(1-\frac 1p)} 2^{-k(d-1)(\frac 2p-1)} N(E\cap I_{k,\mu} , 2^{-j}) ^{\frac 1p} \|g\|_p
\end{multline}
and
\begin{multline} \label{eq:jkmu-2}
2^{-j} \sup_{|s|\le 2^{-j}}
\Big( \sum_{J\in {\Gamma}_{k,\mu,j}} \int_{B_k\setminus B_{k+1} } |\tfrac d{ds} A_{t_J+s} ^j g(x) |^p dx\Big)^{\frac 1p}
\\ {\lesssim} 2^{-j(d-1)(1-\frac 1p)} 2^{-k(d-1)(\frac 2p-1)} N(E\cap I_{k,\mu} , 2^{-j}) ^{\frac 1p} \|g\|_p\,.
\end{multline}
\end{subequations}

We give the proof of \eqref{eq:jkmu-1} and omit the completely analogous proof of \eqref{eq:jkmu-2}.
These inequalities follow from the cases $p=1$ and $p=2$ by interpolation. For $p=1$ we
observe that
$\{x\in B_k: ||x-y|-t_J|{\lesssim} 2^{-j} \}$ has measure $O(2^{-k(d-1)} 2^{-j})$ and $\|K_t^j\|_\infty =O(2^{j})$ for $1\le t\le 2$ and
estimate
\begin{align*} \int_{B_k} |A_{t_J} ^j g(x) | dx \le \int |g(y)| \int_{B_k} |K_{t_J}^j (x-y)|dx \, dy{\lesssim} 2^{-k(d-1)} \|g\|_1.
\end{align*}
Now \eqref{eq:jkmu-1} follows for $p=1$ by summing over the intervals in ${\Gamma}_{k,\mu,j}$. For $p=2$
we use $|\widehat \sigma(\xi)|{\lesssim} (1+|\xi|)^{-(d-1)/2}$ which gives $\|A^j_{t_J}\|_{L^2\to L^2}=O( 2^{-j(d-1)/2} )$ and
\eqref{eq:jkmu-1}
for $p=2$ follows.
Analogous arguments apply to \eqref{eq:jkmu-2} and thus one gets
\eqref{eq:mainunweighted}.

Let \[g_{k,\mu}(y)=g(y) {{\mathbbm 1}}_{2^{-k}\mu\le |y|\le 2^{-k}(\mu+1)} (y)\] and observe that \[ A^j_t g(x) =A^j_t\Big[ \sum_{\nu=\mu-2}^{\mu+2} g_{k,\nu} \Big ] (x) \quad \text{ for $|x|\le 2^{-k}$ and $t\in I_{k,\mu} $.}
\]
Moreover $w_\alpha(x)\approx 2^{-k\alpha}$ for $x\in B_k\setminus B_{k+1}$. Hence,
\begin{align*}
&\Big(\int_{B_k\setminus B_{k+1}} \sup_{t\in E} |A_t^j g(x)|^p|x|^\alpha dx\Big)^{\frac 1p} \\
&{\lesssim} 2^{-k\frac{\alpha}{p}} \Big( \int_{B_k} \sum_{\mu} \sup_{t\in E\cap I_{k,\mu} } \Big| A^j_t \big [\sum_{\nu=\mu-2}^{\mu+2} g_{k,\nu} \big] (x) \Big|^p dx\Big)^{\frac 1p}
\\
&{\lesssim} 2^{-k\frac{\alpha}{p}} 2^{-j(d-1)(1-\frac 1p)} 2^{-k(d-1)(\frac 2p-1)} \sup_\mu N(E\cap I_{k,\mu} , 2^{-j} )^{\frac 1p} \Big(\sum_\mu \Big\|\sum_{\nu=\mu-2}^{\mu+2} g_{k,\nu}\Big\|_p^p\Big)^{\frac 1p}
\\
&{\lesssim} 2^{-j(d-1)(1-\frac 1p)} 2^{-k(\frac{\alpha}{p}+(d-1)(\frac 2p-1))} \sup_\mu N(E\cap I_{k,\mu} , 2^{-j} )^{\frac 1p} \|g\|_p
\end{align*}
and the lemma is proved.
\end{proof}

To procced with the estimate for ${{\mathcal{M}}}_E{f_{\mathrm{mid}}}$ (i.e. the proof of \eqref{eq:goal} we decompose
$ A_t{f_{\mathrm{mid}}} =\sum_{j=0}^\infty A^j_t L_j {f_{\mathrm{mid}}} $
and then the left hand side of \eqref{eq:goal} as
\begin{align*}
&\Big( \int_{2^{-\ell-1} \le |x|\le 2^{-4}} \sup_{t\in E} |A_t {f_{\mathrm{mid}}} (x) |^p w_{1-d}(x) dx\Big)^{1/p} \\
& = \Big( \sum_{k=4}^\ell\int_{2^{-k-1}<|x|\le 2^{-k}} \sup_{t\in E } \Big |\sum_{j=0}^\infty A^j_t L_j {f_{\mathrm{mid}}} (x) \Big |^p w_{1-d}(x) dx\Big)^{1/p} = I+II
\end{align*} where
\begin{subequations}\begin{align}
\label{eq:I}
I&= \Big( \sum_{k=4}^\ell\int_{2^{-k-1}<|x|\le 2^{-k}} \sup_{t\in E } \Big |\sum_{n=0}^\infty A^{k+n}_t L_{k+n} {f_{\mathrm{mid}}} (x) \Big |^p w_{1-d}(x) dx\Big)^{1/p} ,
\\ \label{eq:II}
II&= \Big( \sum_{k=4}^\ell\int_{2^{-k-1}<|x|\le 2^{-k}} \sup_{t\in E } \Big |
\sum_{m=1}^k A^{k-m}_t L_{k-m} {f_{\mathrm{mid}}} (x) \Big |^p w_{1-d}(x) dx\Big)^{1/p} .
\end{align}
\end{subequations}
By Minkowski's inequality $I\le \sum_{n=0}^\infty I_n$, $II\le
\sum_{m=1}^k II_m$ where
\begin{subequations}
\begin{align}\label{eq:In}
I_n&= \Big( \sum_{k=4}^\ell\int_{2^{-k-1}<|x|\le 2^{-k}} \sup_{t\in E(\ell) } \Big | A^{k+n}_t L_{k+n} {f_{\mathrm{mid}}} (x) \Big |^p w_{1-d}(x) dx\Big)^{1/p} ,
\\\label{eq:IIm}
II_m&= \Big( \sum_{ \max\{m,4\}\le k\le \ell }\int_{2^{-k-1}<|x|\le 2^{-k}} \sup_{t\in E } \Big | A^{k-m}_t L_{k-m} {f_{\mathrm{mid}}} (x) \Big |^p w_{1-d}(x) dx\Big)^{1/p} .
\end{align}
\end{subequations}

Let $\chi_k$ be the characteristic function of $\{x:2^{-k-1}\le |x|\le 2^{-k} \}$. Let ${{\mathcal{U}}}=\{y: \frac 14\le |y|\le 8\}$
and let
\begin{align} \label{def:Skn}S^{k,n}_t f(x)&= \chi_k(x) A_t^{k+n} [ {{\mathbbm 1}}_{{\mathcal{U}}} f](x),
\\
\label{def:Tkm}T^{k,m}_t f(x)&= \chi_k(x) A_t^{k-m} [ {{\mathbbm 1}}_{{\mathcal{U}}} f](x).
\end{align}

We prove an inequality for a vector-valued operator.

\begin{proposition}
Let $p> 1+\frac{\gamma}{d-1}$, $0\le {\varepsilon}< (d-1)(p-1)-\gamma$ and $p\le q\le 2<\infty$. Then for $n\ge 0$, $0\le m\le \ell$,
\begin{subequations}
\begin{align}
\label{eq:LpellqS}
&\Big\| \sup_{t\in E} \Big|\sum_{k=4}^\ell S^{k,n}_t f_{k} \Big| \Big\|_{L^p(w_{1-d})} {\lesssim} 2^{-n{\varepsilon}/(2q) }
{{\mathscr C}}(\ell)^{1-p/q}
\Big\|\Big(\sum_{j=0}^\infty |f_j|^q\Big)^{1/q} \Big\|_p ,
\\
\label{eq:LpellqT}
&\Big\| \sup_{t\in E} \Big|\sum_{k=\max\{m,4\}}^\infty T^{k,m}_t f_{k} \Big| \Big\|_{L^p(w_{1-d})} {\lesssim} 2^{-m/q} {{\mathscr C}}(\ell)^{1- p/q}
\Big\|\Big(\sum_{j=0}^\infty |f_j|^q\Big)^{1/q} \Big\|_p .
\end{align}
\end{subequations}

\end{proposition}

\begin{proof}
We use the complex interpolation formula $[L^p(\ell^p), L^p(\ell^\infty)]_\theta = L^p(\ell^q) $ for $p\le q<\infty$ and $\theta= 1- p/q$, see \cite[Ch. 5]{bergh-lofstrom}.
We first prove the case $p=q$ of \eqref{eq:LpellqS}, i.e.
\begin{equation}\label{eq:LpellpS}\Big\| \sup_{t\in E} \Big|\sum_{k=0}^\ell S^{k,n}_t f_{k} \Big| \Big\|_{L^p(w_{1-d})} {\lesssim} 2^{-n{\varepsilon}/(2p) }
\Big(\sum_{j=0}^\infty \|f_j\|_p^p\Big)^{1/p} .
\end{equation}
The left hand side is dominated by
\begin{align*}
&\Big(\sum_{k=0}^\ell \int_{B_k\setminus B_{k+1} } \sup_{t\in E} \big|S^{k,n}_t f_k(x)|^p |x|^{1-d}dx \Big)^{1/p}
\\&{\lesssim} \Big(\sum_{k=0}^\ell
2^{-(k+n)\frac{(d-1)p}{p'} } \sup_{\substack{I\subset [1,2] \\ |I|\approx 2^{-k} } } N(E\cap I,2^{-k-n} ) 2^{-k((d-1)( 2-p)+1-d) } \|f_k\|_p^p \Big)^{1/p}
\end{align*}
by Lemma \ref{lem:k<j}. By the definition of Assouad dimension
\[N(E\cap I,2^{-k-n} ){\lesssim}_{\varepsilon} 2^{n(\gamma+\frac{\varepsilon} 2)}\] when $|I|=2^{-k}$ and therefore we get
\begin{equation*}\Big\| \sup_{t\in E} \Big|\sum_{k=0}^\ell S^{k,n}_t f_{k} \Big| \Big\|_{L^p(w_{1-d})} {\lesssim} 2^{-n(d-1)(p-1)/p} 2^{n(\gamma+\frac {\varepsilon} 2)/p }
\Big\|\Big(\sum_{j=0}^\infty |f_j|^p\Big)^{1/p} \Big\|_p
\end{equation*}
which gives \eqref{eq:LpellpS} by our assumption ${\varepsilon}<(d-1)(p-1)-\gamma$.

For $q=\infty$ we prove
\begin{equation}
\label{eq:LpellinftyS}
\Big\| \sup_{t\in E} \Big|\sum_{k=0}^\infty S^{k,n}_t f_{k} \Big| \Big\|_{L^p(w_{1-d})} {\lesssim} {{\mathscr C}}(\ell)
\big\|\sup_j|f_j| \big\|_p .
\end{equation}
We have the pointwise estimate (by the positivity property of the spherical mean operator)
\begin{align*}
\sup_{t\in E}\Big|\sum_{k=0}^\ell S^{k,n}_t f_{k}(x) \Big|\le \sup_{t\in E} A_t \big[\sup_k |L_{k+n} [{{\mathbbm 1}}_{{\mathcal{U}}} f_k]|\big]
\end{align*} and therefore
\begin{align*}& \Big\| \sup_{t\in E} \Big|\sum_{k=0}^\ell S^{k,n}_t f_{k} \Big| \Big\|_{L^p(w_{1-d})} \\& {\lesssim} {{\mathscr C}}(\ell) \big\| \sup_k |L_{k+n} [{{\mathbbm 1}}_{{\mathcal{U}}} f_k] | \big\|_{L^p(w_{1-d})}
{\lesssim} {{\mathscr C}}(\ell) \big\| M_{HL} [\sup_k | {{\mathbbm 1}}_{{\mathcal{U}}} f_k | ]\big\|_{L^p(w_{1-d})}
\\& {\lesssim} {{\mathscr C}}(\ell) \big\| {{\mathbbm 1}}_{{\mathcal{U}}} \sup_k|f_k| \big\|_{L^p(w_{1-d})}
{\lesssim} {{\mathscr C}}(\ell) \big\| \sup_k|f_k| \big\|_{p} .
\end{align*}
Here we have used the definition of ${{\mathscr C}}(\ell)$, the boundedness of the Hardy-Littlewood maximal operator on $L^p(w_{1-d})$, and the fact that $|y|^{1-d} \approx 1$ for $y\in {{\mathcal{U}}}$. This proves \eqref{eq:LpellinftyS}.
We now interpolate between \eqref{eq:LpellpS} and \eqref{eq:LpellinftyS} to get \eqref{eq:LpellqS}.

The estimates for $T^{k,m}_t$ are analogous if instead of Lemma \ref{lem:k<j} we use Lemma \ref{lem:k>j}.
\end{proof}
\begin{proof}[Proof of Theorem \ref{thm:newfullga}, concluded]
We shall now derive for $I$, $II$ as in \eqref{eq:I}, \eqref{eq:II} the estimate
\begin{equation}\label{eq:I-II} I+II \le C(p) (1+ {\varepsilon}^{-1}) {{\mathscr C}}(\ell) ^{1-p/2} \|f\|_{L^p(w_{1-d})} .
\end{equation}

For this we apply \eqref{eq:LpellqS} with $q=2$ and with $f_k(y) = L_{k+n} {f_{\mathrm{mid}}}(y)$, observing that ${{\mathbbm 1}}_{{\mathcal{U}}} L_{k+n}f_k=f_k$
by the support properties of ${f_{\mathrm{mid}}}$ and the kernels of the $ L_{k+n}$. Thus we get
\begin{align*}
I_n&{\lesssim} 2^{-n {\varepsilon}/4} {{\mathscr C}}(\ell)^{1-p/2} \Big\| {{\mathbbm 1}}_{{\mathcal{U}}} \Big(\sum_{k=0} ^\infty| L_{k+n} {f_{\mathrm{mid}}} |^2\Big)^{1/2} \Big\|_p
\\
&{\lesssim} 2^{-n {\varepsilon}/4} {{\mathscr C}}(\ell)^{1-p/2} \Big\| \Big(\sum_{k=0} ^\infty| L_{k+n} {f_{\mathrm{mid}}} |^2\Big)^{1/2} \Big\|_{L^p(w_{1-d)} }
\\
&{\lesssim}_p 2^{-n {\varepsilon}/4} {{\mathscr C}}(\ell)^{1-p/2}
\|f\|_{L^p(w_{1-d}) }
\end{align*}
by Littlewood-Paley theory in $L^p(w_{1-d})$ (\cite{Kurtz80}). After summing in $n$, we get \eqref{eq:I-II} for the term $I$ and the corresponding estimate for $II$ is analogous.

Finally, combining \eqref{eq:lowconst}, \eqref{eq:highconst}, \eqref{eq:midconst} with \eqref{eq:I-II} we get
\[ \big\|\sup_{t\in E(\ell)} |A_t f| \big\|_{L^p(w_{1-d})} \le C(p) (1+ {\varepsilon}^{-1} {{\mathscr C}}(\ell)^{1-p/2} )\|f\|_{L^p(w_{1-d})} \]
and hence, by definition of ${{\mathscr C}}(\ell)$,
\[ {{\mathscr C}}(\ell) \le C(p) (1+ {\varepsilon}^{-1} {{\mathscr C}}(\ell)^{1-p/2}). \]
This implies the $\ell$-uniform bound \[{{\mathscr C}}(\ell) {\lesssim}
\max\big\{ 2C(p), (2C(p){\varepsilon}^{-1})^{2/p}\big\}
\]
as can be seen by separately considering the cases $1\le {\varepsilon}{{\mathscr C}}(\ell)^{1-p/2}$ and $1\ge {\varepsilon}{{\mathscr C}}(\ell)^{1-p/2}$.
The global estimate is similar, making use by a reduction in \cite{DuoandikoetxeaSeijo} which is valid for negative $\alpha$.
\end{proof}

\begin{remarka}
The argument at the end of the proof can be found in a number of harmonic analysis papers and was first attributed to M. Christ in
\cite{CarberyConvexBody}. It was also used by Duoandikoetxea and Vega \cite{DuoandikoetxeaVega} for the lacunary maximal function and has been useful in many other problems in harmonic analysis (among them \cite{Christ-frac, GreenleafSeegerWainger, SeegerWainger-frac, LabaPramanik-Duke, BakOberlinSeegerAJM}).
The Christ argument was designed to simplify the Nagel--Stein--Wainger bootstrap reasoning \cite{NagelSteinWainger-lac} for Hilbert transforms along curves; indeed that bootstrap argument could also be used above.
\end{remarka}
\section{Bounds for the sequence examples}
\subsection{\texorpdfstring{$L^p(w_\alpha)$-boundedness of $M_{E^\mu}$ for $p(\mu)<p<\frac{d}{d-1}$, $\alpha=L(p)$}{Boundedness for minimal alpha}}
This is very similar to the proof of Theorem \ref{thm:newfullga}.
For
$p(\mu)<p<\frac{d}{d-1}$ the value $L(p)$ associated with $E^\mu$ is
\[L(p)\equiv\alpha(p,\mu):=1-d+\frac{d-(d-1)p}{\mu} \,.\]

We now let, for $\ell\ge 4$, ${{\mathscr C}}(\ell)$ be the infimum over all constants $C$ such that
\[ \big\|\sup_{t\in E^\mu} |A_t f | \big \|_{L^p(\Bellc,w_{\alpha(p,\mu)})} \le C \|f\|_{L^p(w_{\alpha(p,\mu)})}\]
for all $f\in L^p(w_{\alpha(p,\mu)})$ .

The inequalities \eqref{eq:lowconst}, \eqref{eq:highconst} and \eqref{eq:midconst} have obvious analogs with $w_{\alpha(p,\mu)}$ in place of $w_{1-d}$.

It is our task to show the analogue of \eqref{eq:goal}, that is
\begin{equation}\label{eq:goalseq}
\Big( \int_{2^{-\ell-1}\le |x|\le 2^{-4}} \sup_{t\in E^\mu } |A_t {f_{\mathrm{mid}}} (x) |^p w_{\alpha(p,\mu)} (x) dx\Big)^{1/p} {\lesssim} \|f\|_{L^p(w_{\alpha(p,
\mu)})}
\end{equation} holds with an implicit constant independent of $\ell$.
We will use again Lemma \ref{lem:k<j}.

If $|I|= 2^{-k}$ and $j\ge k$ then
\[ N(E\cap I, 2^{-j} ){\lesssim} \begin{cases}
2^{j-k} &\text{ if $\frac{\mu}{\mu+1} j\le k \le j$,}
\\
2^{j/(\mu+1)} &\text{ if $0<k\le \frac{\mu}{\mu+1} j$,}
\end{cases} \]
which is essentially sharp at least for $I=[1, 1+2^{-k}].$

Then the quantity in Lemma \ref{lem:k<j},
\[\Lambda_{k,j}^E(p,\alpha) = 2^{-j (d-1)(p-1)} \sup_{I:|I|=2^{-k}} N(E\cap I) 2^{-k ((d-1)(2-p)+\alpha }\] satisfies for $\alpha\ge \alpha(p,\mu)$
\begin{multline}\label{eq:Lambdabounds} \Lambda_{k,j}^E(p,\alpha) \\ {\lesssim} 2^{-k (\alpha-\alpha(p,\mu))}\times
\begin{cases} 2^{ (j- k\frac{\mu+1}{\mu}) (d-(d-1)p)}
&\text{ if $\frac{\mu}{\mu+1} j\le k \le j$}\,,
\\
2^{(k \frac{\mu+1}\mu-j) ((d-1)(p-1)-\frac{1}{\mu+1}) }
&\text{ if $0<k\le \frac{\mu}{\mu+1} j$}\,.
\end{cases}
\end{multline}

For $p(\mu)<p<\frac{d}{d-1}$ the displayed estimate means that for fixed $k$ the quantity $j\mapsto {\Lambda}_{j,k}$ peaks at $j=\frac{\mu+1}{\mu} k$ (where it is $O(1)$) and has exponential decay $2^{-{\epsilon}(p)| j-\frac{\mu+1}{\mu} k|} $ with ${\epsilon}(p)>0$ as $j$ moves away from $\frac{\mu+1}{\mu} k$.

Let $n\in {{\mathbb{Z}}}$ and $j_n(k)$ be the unique integer such that \[ \tfrac{\mu+1}\mu k+n\le j<\tfrac{\mu+1}{\mu} k+n+1\]
Let ${{\mathcal{K}}}(n,\ell)=\{k: 4\le k\le \ell, j_n(k)\ge 0\}$. For $k\in {{\mathcal{K}}}(n,\ell)$,
\begin{equation} \label{eq:Smoddef} {{\mathcal{S}}}^{k,n} f(x) = \chi_k(x) A_t^{j_n(k)} [{{\mathbbm 1}}_{{\mathcal{U}}} f](x).
\end{equation}

We have takes following inequality that takes the place of \eqref{eq:LpellqS} in the present situation: For
$p< q< \infty$:
\begin{equation}
\label{eq:LpellqSmod}
\Big\| \sup_{t\in E} \Big| \sum_{k\in {{\mathcal{K}}}(n,\ell)} {{\mathcal{S}}}^{k,n}_t f_k\Big| \Big\|_{L^p(w_{\alpha(p,\mu)})} {\lesssim} 2^{-|n|{\varepsilon}/ 2q} {{\mathscr C}}(\ell) ^{1-p/q} \Big\|\Big(\sum_{k=0}^\infty|f_k|^q\Big)^{1/q} \Big\|_p
\end{equation}
for some ${\varepsilon}>0$; with
$0<{\varepsilon}<\frac{\mu}{\mu+1}(d-(d-1)p)$ when $n<0$ and $0<{\varepsilon} <\frac{\mu}{\mu+1} ((d-1)(p-1)-\frac{1}{\mu+1})$. It is possible to make both choices simultaneously since $p(\mu)<p<\frac{d}{d-1}$.

\eqref{eq:LpellqSmod} which follows from
\begin{equation}
\label{eq:LpellpSmod}
\Big\| \sup_{t\in E} \big| \sum_{k\in {{\mathcal{K}}}(n,\ell)} {{\mathcal{S}}}^{k,n}_t f_k\big|\Big\|_{L^p(w_{\alpha(p,\mu)})} {\lesssim} 2^{-|n|{\varepsilon}/ 2p} \Big\|\Big(\sum_{k=0}^\infty |f_k|^p\Big)^{1/p} \Big\|_p
\end{equation}
and
\begin{equation}
\label{eq:LpellinftySmod}
\Big\| \sup_{t\in E} \big| \sum_{k\in {{\mathcal{K}}}(n,\ell)} {{\mathcal{S}}}^{k,n}_t f_k\big|\Big\|_{L^p(w_{\alpha(p,\mu)})} {\lesssim} {{\mathscr C}}(\ell)\big\|\sup_{k\ge 0}|f_k|\big\|_p
\end{equation}
The proof of \eqref{eq:LpellinftySmod} is analogous to the proof of \eqref{eq:LpellinftyS}.
Given \eqref{eq:Lambdabounds} the proof of \eqref{eq:LpellpSmod} is analogous to the proof of \eqref{eq:LpellpS}.

\subsection{\texorpdfstring{$L^p(w_{1-d})$ unboundedness for $p=\frac{d}{d-1}$}{Unboundedness for critical p}}\label{sec:lowerboundpd}

Consider the case \[ E_\mu=\{ 1+n^{-\mu}: n\in {{\mathbb{N}}}\} \text{ for $\mu<1$.}\]
We modify an example in \cite[\S4]{FraccaroliRoosSeeger}.

We fix $k$ large.
Let \[R_j = \{ (y',y_d) : |y_d-1|\le {\varepsilon}^{-1}2^{-j}, \, {\varepsilon} 2^{k-j}\le |y'|\le {\varepsilon}^{-1} 2^{k-j} \} \] where we put the restriction $2k\le j\le k\frac{\mu+1}{\mu} $ (this gives a range of $j$ of cardinality $\approx k$ since $\mu<1$. Note that $R_j$ is comparable to the outer part of a box with $(d-1)$ long horizontal sides of length $2^{k-j}$ and one short vertical side of length $2^{-j}$ around the point $(0,1)$. Let \[ g_j= 2^{(j-k)\frac{d-1}{p}} 2^{j/p} {{\mathbbm 1}}_{R_j}. \]Then $\|g_j\|_p =1$. Below we will apply this for $p=\frac{d}{d-1}$.

Let ${\epsilon}\ll \frac{1-\mu}{\mu}$ and ${\Gamma}_k=\{ j: 2k(1+{\epsilon})\le j \le k(\frac{\mu+1-{\epsilon}}{\mu})\} $ so that $\#\Gamma_k\ge ck.$
Let \[ f_k= \sum_{j\in \Gamma_k} g_j.\]
Since every $y$ belongs to at most $O({\varepsilon}^{-1})$ of the $R_j$ we observe that \begin{equation}\label{eq:normfk} \|f_k\|_{L^p(w_{1-d})}\approx \|f_k\|_p {\lesssim}_{{\varepsilon},\mu} k^{1/p} .\end{equation}

Let \[ V_k=\{(x',x_d): -2^{-k+1}\le x_d \le -2^{-k-1}, \,\, |x'|\le 2^{-k}\}.\] There exists $n$ such that $n^{-\mu} \approx 2^{-k}$ and $|n^{-\mu}+x_d|\le 2^{-k\frac{\mu+1}{\mu}}$ (and note that for $j\in \Gamma_k$ we have $2^{-k\frac{\mu+1}{\mu}}\ll 2^{-j}$ for large $k$).
Then for $j\in \Gamma_k$, $t_n=1+n^{-\mu}$ we have
\[ A_{t_n} g_j(x) {\gtrsim} \text{width}(R_j) 2^{(j-k)\frac{d-1}{p}} 2^{j/p} {\gtrsim} 2^{(k-j)(d-1)} 2^{jd/p} 2^{-k\frac{d-1}{p}}
\] This gives a meaningful lower bound for $p=\frac{d}{d-1} $:
\[ A_{t_n} g_j(x) \ge 2^{k\frac{d-1}{d}} \text{ for all } j\in \Gamma_k.\] Hence
\[ A_{t_n} f_k(x) {\gtrsim} k 2^{k\frac{d-1}{d}} \]
Since $x\in V_k$ was arbitrary we obtain
\[\sup_{n: n^{-\mu} \approx 2^{-k} } A_{t_n} f_k(x) {\gtrsim} k 2^{k\frac{d-1}{d}}\] and thus
\begin{equation} \label{eq:Mfknorm}
\Big(\int_{V_k} |M_{E_\mu} f_k(x)|^{\frac{d}{d-1} }|x|^{1-d} dx\Big)^{\frac{d-1}{d} } {\gtrsim} k.
\end{equation}
We thus see from \eqref{eq:normfk} and \eqref{eq:Mfknorm} that
\[ \frac{ \|M_{E_\mu} f_k\|_{L^{\frac{d}{d-1}}(w_{1-d})} }{\|f_k\|_{L^{\frac{d}{d-1}}(w_{1-d})} } {\gtrsim} k^{\frac{1}{d-1}} \]
which yields the unboundedness of $M_E$ on $L^p(w_{1-d})$.

We fix $k$ large. Let
\[ R^j=\{ (y',y_d) : |y'|\leq \varepsilon^{-1}2^{k-j},\,\, 1-\varepsilon^{-1}2^{-j}\leq |y'|^2+|y_d|^2\leq 1+\varepsilon^{-1}2^{-j} \} \]
where we put the restriction $k\leq j\leq \min\{ 2k,\frac{\mu+1}{\mu}k \}$. Note that for any $\mu>0$, we have the cardinality of $j$ is $\approx k$ where the implicit constant depending on $\mu$. Let
\[ h_j=2^{(j-k)\frac{d-1}{p}+\frac{j}{p}}{{\mathbbm 1}}_{R^j} \]
so that $\Vert h_j\Vert_{L^p(w_{1-d})}\approx 1$. Define $\Gamma_k=\{ j: 2k(1+{\epsilon})\leq j\leq k(\frac{\mu+1}{\mu}-{\epsilon}) \}$ for a sufficiently small ${\epsilon}>0$ so that $\sharp \Gamma_k\geq ck$. Then, we define
\[ f_k=\sum_{j\in\Gamma_k}h_j. \]
By the triangle inequality, one can easily check that $\| f_k\|_{L^p(w_{1-d})}\lesssim k^{\frac{1}{p}}$ holds for any $p>1$.

Now define
\[ V^k=\{ (x',x_d) : -2^{-k+1}\leq x_d\leq -2^{-k-1},\,\,|x'|\leq 2^{-k} \}. \]
We now claim that for any $-2^{k+1}\leq x_d\leq -2^{-k-1}$, there exists $n$ such that
\begin{equation}\label{eq:finding good n}
n^{-\mu}\approx 2^{-k} \text{ and } |n^{-\mu}+x_d|\leq 2^{-k\frac{\mu+1}{\mu}}.
\end{equation}
Assuming \eqref{eq:finding good n}, one can see that for $t_n=1+n^{-\mu}$ and $j\in\Gamma_k$,
\[ A_{t_n}h_j(x)\gtrsim \text{area}(R_j)2^{(j-k)\frac{d-1}{p}+\frac{j}{p}} {\gtrsim} 2^{(k-j)(d-1) }2^{jd/p} 2^{-k\frac{d-1}{p}}.\]
Hence, for $p = \frac{d}{d-1}$,
\[A_{t_n}f_k(x) {\gtrsim} k 2^{k \frac{d-1}{d}}.\]
For $x \in V_k$, we have
\[M_{E_\mu} f_k (x) {\gtrsim} k 2^{k \frac{d-1}{d}}.\]
Thus
\[\Big(\int_{V_k} |M_{E_\mu} f_k(x)|^{\frac{d}{d-1} }|x|^{1-d} dx\Big)^{\frac{d-1}{d} } {\gtrsim} k\]
and
\[ \frac{ \|M_{E_\mu} f_k\|_{L^{\frac{d}{d-1}}(w_{1-d})} }{\|f_k\|_{L^{\frac{d}{d-1}}(w_{1-d})} } {\gtrsim} k^{\frac{1}{d}}. \]

\subsection{\texorpdfstring{Unboundedness of ${{\mathcal{M}}}_{{{\mathcal{E}}}^\mu}$ on $L^p(\omega_\alpha)$ for $\alpha=U(p)$}{Unboundedness for maximal alpha}}\label{unbdd U(p)}

In this section, we show that the global maximal operator ${{\mathcal{M}}}_{{{\mathcal{E}}}^\mu}$ can never be bounded on $L^p({\omega}_\alpha)$ for the upper endpoint $\alpha=U(p)$. Thus, we completely characterize the boundedness of ${{\mathcal{M}}}_{{{\mathcal{E}}}^\mu}$ on $L^p(\omega_\alpha)$. Note that $U(p)=(d-1)(p-1)-\frac{1}{\mu+1}$.

Define $f=\chi_{B(0,1)}$, the characteristic function of a ball of radius 1 centered at the origin. Since $\alpha=U(p)\geq 0$, we have $\| f\|_{L^p(\omega_\alpha)}\approx 1$. Now we claim
\begin{equation}
\| {{\mathcal{M}}}_{{{\mathcal{E}}}^\mu}f\|_{L^p({\omega}_\alpha)}\gtrsim (\sum_{k\geq 0} 2^{k(\alpha-(d-1)(p-1)+\frac{1}{\mu+1})})^{\frac{1}{p}}.
\end{equation}
This shows that when $\alpha=U(p)$, the operator ${{\mathcal{M}}}_{{{\mathcal{E}}}^\mu}$ is unbounded on $L^p({\omega}_\alpha)$ for any $p$.

For $x\in \{ y: |y|\in 2^k(1+n^{-\mu})+(-\frac{1}{100},\frac{1}{100})\}=:\mathfrak{S}(k,n)$, one can see that
\[ |{{\mathcal{M}}}_{{{\mathcal{E}}}^\mu}f(x)|\gtrsim 2^{-k(d-1)} \]
holds. For any fixed sufficiently large $k$, we show that $\mathfrak{S}(k,n)$ are disjoint for $1\leq n\leq 2^{\frac{k}{\mu+1}}\mu^{\frac{1}{\mu+1}}$. It suffices to show that intervals $\mathfrak{I}(k,n):=2^k(1+n^{-\mu})+(-\frac{1}{100},\frac{1}{100})$ are disjoint. By the mean value theorem, we have
\[ 2^k((1+(n-1)^{-\mu})-(1+n^{-\mu}))\geq 2^k\mu n^{-\mu-1}\geq 1 \]
for $2\leq n\leq 2^{\frac{k}{\mu+1}}\mu^{\frac{1}{\mu+1}}$. Therefore, we obtain
\[ \left|\bigcup_{n=1}^{2^{\frac{k}{\mu+1}}\mu^{\frac{1}{\mu+1}}}\mathfrak{S}(k,n)\right|\gtrsim 2^{\frac{k}{\mu+1}}2^{k(d-1)}. \]
since $|x|\sim 2^k$ for $x\in\mathfrak{S}(k,n)$, we finally get
\[ \| {{\mathcal{M}}}_{{{\mathcal{E}}}^\mu}f\|^p_{L^p({\omega}_\alpha)}\gtrsim \sum_{k\geq 0}2^{k\alpha}2^{-k(d-1)p}2^{\frac{k}{\mu+1}}2^{k(d-1)}= \sum_{k\geq 0} 2^{k(\alpha-(d-1)(p-1)+\frac{1}{\mu+1})}, \]
as desired.

\subsection{\texorpdfstring{Unboundedness of $M_{[1,2]}$ on $L^p(w_{U(p)})$}{Unboundedness of the full maximal function for maximal alpha}}

Here $U(p)=(d-1)(p-1)-1$ for $p>\frac{d}{d-1}$.
We use the familiar example due to Stein \cite{SteinPNAS1976}.
Let $f(y) = |y|^{1-d} [\log(\tfrac{1}{|y|})]^{-u} {{\mathbbm 1}}_{|y|\le 1/2}$. Then
\[ \int |f(y)|^p |y|^{(d-1)(p-1)-1} dy= \int_{|y|<1/2} |y|^{-d} (\log\frac 1{|y|} )^{-u p} dy\]
which is finite for $u>1/p.$
However if $u<1$ then $A_t f(x)=\infty$ for $1<|x|<2$ and thus $M_{[1/2]} f\notin L^p(w_{U(p)}).$
Thus we obtain a unboundedness by choosing $u\in (\frac 1p,1)$.

Following \cite{SeegerWaingerWright1997} one can derive counterexamples for other $E\subset [1,2]$ with $|E|=0$ and $\beta=1$, following \cite{SeegerWaingerWright1997}.
If one defines $g_\delta (y) =|y|^{1-d} {{\mathbbm 1}}_{\delta\le |y| \le 1/2}$,
the \[\Big(\int |g_\delta(y) |^p |y|^{(d-1)(p-1) -1} dy\Big)^{1/p} {\gtrsim} [\log \tfrac 1\delta]^{1/p} \] while
$A_{t} g_\delta(x) {\gtrsim} \log \frac 1 \delta$ if $t\in E$ and $||x|-t|<c\delta$. This means that $M_E g_\delta(x){\gtrsim} \log 1/\delta $ on a subset of $\{1\le |x|\le 2\}$ which is of measure ${\gtrsim} N(E,\delta) \delta $. This leads to the necessary condition
$N(E,\delta){\lesssim} \delta^{-1} [\log \tfrac 1\delta]^{1-p}$ uniformly in $\delta$.

\subsection{\texorpdfstring{Boundedness for $M_{E^\mu}$ when $\alpha=U(p)$}{Boundedness of local operator for maximal alpha}} Here $U(p)=(d-1)(p-1)-\frac{1}{\mu+1}.$
In contrast to the situation for the global operator with respect to ${{\mathcal{E}}}^\mu$, the local maximal operator $M_{E^\mu}$ is bounded when $\alpha=U(p)$.
We need an elementary application of the Schur test.

\begin{lemma} \label{lem:Gamma}
Let $\phi\in L^\infty$, with support in $\{x:|x|\le 1\}$ and let $\chi_k$ be the characteristic function of $\{y\in {{\mathbb{R}}}^d: 2^{-k-1}\le |y|\le 2^{-k}\}.$ For ${{\mathfrak{z}}}\in {{\mathbb{Z}}}$, $m\ge 0$ define the integral operator $\Gamma=\Gamma_{\rho,m,{{\mathfrak{z}}}}$ by
\[\Gamma f(x)= \int 2^{md}\phi(2^m (x-y)) |y|^{-\rho} \chi_{m+{{\mathfrak{z}}}}(y) f(y) dy\]
Then, for $1\le p\le \infty$,
\[\|\Gamma\|_{L^p\to L^p} {\lesssim} \begin{cases} 2^{(m+{{\mathfrak{z}}})\rho} 2^{-{{\mathfrak{z}}} d/p'} &\text{ if ${{\mathfrak{z}}}\ge 0$}
\\2^{(m+{{\mathfrak{z}}}) \rho } &\text{ if ${{\mathfrak{z}}}\le 0$}
\end{cases} \]
\end{lemma}
\begin{proof} This follows from Schur's test. Let $K(x,y)= 2^{md}\phi(2^m (x-y)) |y|^{-\rho} \chi_{m+{{\mathfrak{z}}}}(y) $.
Then
\[ \sup_y\int |K(x,y)| dx {\lesssim} 2^{-(m+{{\mathfrak{z}}})(-\rho)}\]
and
\[\sup_x\int |K(x,y)| dy {\lesssim} \begin{cases} 2^{md } 2^{-(m+{{\mathfrak{z}}})d } 2^{-(m+{{\mathfrak{z}}})(-\rho)} &\text{ if } {{\mathfrak{z}}}\ge 0
\\ 2^{-(m+{{\mathfrak{z}}})(-\rho)} &\text{ if } {{\mathfrak{z}}}\le 0
\end{cases} \] Hence $\|\Gamma\|_{L^1\to L^1} {\lesssim} 2^{(m+{{\mathfrak{z}}})\rho}$. For the $L^\infty \to L^\infty $ bound we get $\|\Gamma\|_{L^\infty\to L^\infty} {\lesssim} 2^{-{{\mathfrak{z}}} d} 2^{(m+{{\mathfrak{z}}})\rho}$ for ${{\mathfrak{z}}} \ge 0$ and
$\|\Gamma\|_{L^\infty\to L^\infty} {\lesssim} 2^{(m+{{\mathfrak{z}}})\rho}$ for ${{\mathfrak{z}}} \le 0$. Now apply interpolation.
\end{proof}

We will apply the lemma with $\rho= \frac{U(p)}p= \frac{d-1}{p'}- \frac 1{(\mu+1)p}$.
The goal is to show
\[ \Big( \int \sup_{t\in E^\mu} |A_t g(x)|^p w_\alpha(x) dx\Big)^{1/p} {\lesssim} \|f\|_{L^p(w_\alpha)}\] for $\alpha= U(p)$ and $p> 1+ (\mu+1)^{-1}(d-1)^{-1}$.

The inequalities
\begin{align*} \Big( \int_{|x|<\frac 14} \sup_{t\in E^\mu} |A_t g(x)|^p w_\alpha(x) dx\Big)^{1/p} &{\lesssim} \|g\|_{L^p(w_\alpha)}\\
\Big( \int_{|x|>4} \sup_{t\in E^\mu} |A_t g(x)|^p w_\alpha(x) dx\Big)^{1/p} &{\lesssim} \|g\|_{L^p(w_\alpha)}
\end{align*} follow from the known unweighted estimates, for all $\alpha$.

It remains to show
\begin{equation}
\Big( \int_{\frac 14\le|x|\le 4} \sup_{t\in E^\mu} |A_t g(x)|^p w_\alpha(x) dx\Big)^{1/p} {\lesssim} \|g\|_{L^p(w_\alpha)}
\end{equation}
Setting $f(y)= g(y)|y|^{\alpha/p} $ this is equivalent with
\begin{equation}
\Big( \int_{\frac 14\le|x|\le 4} \sup_{t\in E^\mu} |A_t [f|\cdot|^{-\alpha}](x)|^p dx\Big)^{1/p} {\lesssim} \|f\|_{L^p},
\end{equation}
with unweighted $L^p$ norms. In what follows we set
\[ v_p(y):= |y|^{-\frac{U(p)}{p}} \equiv |y|^{\frac{1}{(\mu+1)p}-\frac{d-1}{p'}}.\]
Instead of writing $\frac 14\le |x|\le 4$ we will write $|x|\sim 1$.

We may assume that $f$ is supported in $\{|y|\le \frac 1{100}\}$, otherwise we can again use the unweighted bounds.

Let \[ b_l= 1+2^{- l \frac{\mu}{\mu+1}}, \qquad I_l = [b_{l+1}, b_l] \]
The sequence $l\mapsto b_l-1$ is lacunary.
Note that the dilation set $E^\mu\cap I_l$ is essentially equally spaced, with $2^{-l}$ gaps, i.e. $|n^{-\mu}-(n+1)^{-\mu} |\approx 2^{-l} $ if $n^{-\mu}\in I_l$.

We can write (at least a priori for sufficiently smooth $f$)
\begin{subequations}
\begin{align}
\notag&\sup_{t\in E^\mu} |A_t f(x)|=\sup_l \sup_{t\in E^\mu\cap I_l} \Big| \sum_{j=0}^\infty A_t^j L_j(f v_p ) \Big|
\\
\label{eq:Ia}
&\quad\le \sup_l \sup_{t\in E^\mu\cap I_l} \Big| \sum_{j=l}^\infty A_t^j L_j(f v_p ) \Big|
\\
\label{eq:IIa}
&\quad+
\sup_l \sup_{t\in E^\mu\cap I_l} \Big| \sum_{l\frac{\mu}{\mu+1} <j\le l} A_t^j L_j(f v_p ) \Big|
\\
\label{eq:III}
&\quad+
\sup_l \sup_{t\in E^\mu\cap I_l} \Big| \sum_{0\le j\le l\frac{\mu}{\mu+1} } \big( A_t^j L_j(f v_p ) - A_{b_{l}}^j L_j (fv_p) \big) \Big|
\\
\label{eq:IV}
&\quad +
\sup_l \sup_{t\in E^\mu\cap I_l} \Big| \sum_{0\le j\le l\frac{\mu}{\mu+1} } A_{b_{l}}^j L_j (fv_p) \Big|
\end{align}
\end{subequations}
For all $\alpha< (d-1)(p-1)$ the terms \eqref{eq:III} and \eqref{eq:IV} satisfy better estimates, due to the lacunarity of the sequence $(b_l-1)_{l\ge 0} $.
\begin{lemma} For $1<p<\infty$, $\alpha<(d-1)(p-1)$
\begin{align}
\Big\|\sup_l \sup_{t\in E^\mu\cap I_l} \Big| \sum_{0\le j\le l\frac{\mu}{\mu+1} } A_{b_{l}}^j L_j (fv_p) \Big|\Big\|_{L^p(|x|\sim 1)}{\lesssim} \|f\|_p
\\
\Big\| \sup_l \sup_{t\in E^\mu\cap I_l} \Big| \sum_{0\le j\le l\frac{\mu}{\mu+1} } \Big( A_t^j L_j(f v_p ) - A_{b_{l}}^j L_j (fv_p) \Big) \Big|\Big\|_{L^p(|x|\sim 1)} {\lesssim} \|f\|_p
\end{align}
\end{lemma}

We thus focus on the terms \eqref{eq:Ia} and \eqref{eq:IIa}.
\begin{lemma}
For $p(\mu)<p<\infty $,
\begin{equation}\label{eq:Iaest}\Big\|
\sup_l \sup_{t\in E^\mu\cap I_l} \Big| \sum_{j=l}^\infty A_t^j L_j(f v_p ) \Big|
\Big\|_{L^p(|x|\sim 1)} {\lesssim} \|f\|_p
\end{equation}
\end{lemma}
\begin{proof}
Observe that for $\eqref{eq:Ia}$ we may estimate
\begin{align*} &\Big\|
\sup_l \sup_{t\in E^\mu\cap I_l} \Big| \sum_{j=l}^\infty A_t^j L_j(f v_{p} ) \Big|
\Big\|_{L^p(|x|\sim 1)}
\\&{\lesssim} \sum_{n=0}^\infty \sum_{{{\mathfrak{z}}}\in {{\mathbb{Z}}}}
\Big\|\sup_{l\ge 0} \big|
\sup_{t\in E^\mu\cap I_l} A_t^{l+n} L_{l+n} (f v_p \chi_{l+n+{{\mathfrak{z}}}}) \big|\Big\|_{L^p(|x|\sim 1)}
\\
&{\lesssim} \sum_{n=0}^\infty \sum_{{{\mathfrak{z}}}\in {{\mathbb{Z}}}}
\Big(\sum_{l=0}^\infty \sum_{t\in E^\mu \cap I_l}
\|A_t^{l+n} L_{l+n} (f v_p \chi_{l+n+{{\mathfrak{z}}}})\|_p^p\Big)^{1/p}
\end{align*}
Now for $1\le p\le 2$,
\begin{multline}
\|A_t^{l+n} L_{l+n} (f v_p \chi_{l+n+{{\mathfrak{z}}}})\|_p {\lesssim} 2^{-(l+n)(d-1)/p'} \| L_{l+n} (f v_p \chi_{l+n+{{\mathfrak{z}}}})\|_p \\ {\lesssim} 2^{-(l+n)(d-1)/p'}2^{(l+n+{{\mathfrak{z}}}) \frac{U(p)}{p}} \min\{ 1, 2^{-{{\mathfrak{z}}} d/p'} \}\|f\chi_{l+n+{{\mathfrak{z}}}}\|_p,
\end{multline}
by Lemma \ref{lem:Gamma}. Also notice that $\#(E^\mu\cap I_l) {\lesssim} 2^{l/(\mu+1)} $.
Thus we obtain for $1\le p\le 2$, with $\alpha=U(p)= (d-1)(p-1)-\frac{1}{\mu+1}$,
\begin{align}
\notag&\Big(\sum_{l=0}^\infty \sum_{t\in E^\mu \cap I_l}
\|A_t^{l+n} L_{l+n} (f v_p \chi_{l+n+{{\mathfrak{z}}}})\|_p^p\Big)^{1/p}
\\ \notag &{\lesssim}
\Big(\sum_{l=0}^\infty \sum_{t\in E^\mu \cap I_l}
2^{-(l+n)(d-1)(p-1)}2^{(l+n+{{\mathfrak{z}}}) U(p)} \min\{ 1, 2^{-{{\mathfrak{z}}} dp/p'} \} \|f\chi_{l+n+{{\mathfrak{z}}}}\|_p^p \Big)^{1/p}
\\ \label{eq:ple2part}
&{\lesssim}\min\{ 2^{-{{\mathfrak{z}}} (\frac 1{p'} + \frac{1}{(\mu+1)p})}, 2^{{{\mathfrak{z}}}\frac{U(p)}p} \} 2^{-\frac{n}{(\mu+1)p}} \Big(\sum_{l=0}^\infty \|f\chi_{l+n+{{\mathfrak{z}}}}\|_p^p \Big)^{1/p} .
\end{align}
We take this result for $p=2$ and interpolate with an inequality for $p=\infty$. Namely $A_t (f|\cdot|^{1-d}\chi_{m+n} )=O(1)$ (since the spherical integral is over a set of diameter ${\lesssim} 2^{-m-n}$) and thus
\[\sup_{l\ge 0} \sup_{t\in E^\mu \cap I_l}
\sup_x|A_t^{l+n} L_{l+n} (f |\cdot|^{1-d} \chi_{l+n+{{\mathfrak{z}}}})(x)| {\lesssim} \sup_l \|f \chi_{l+n+{{\mathfrak{z}}}})\|_\infty\]
at least for continuous bounded $f$.
Hence, by an interpolation with analytic families argument, using for $0\le {\operatorname{Re\,}}(z)\le 1$
\[ A_t [f|\cdot|^{-(d-1)(1-\frac{1-z}2)}]=\int f(x+t{\omega}) |x+t{\omega}|^{(1-d)(\frac 12+\frac 1{2z})} d\sigma({\omega}),\]
we obtain for $2\le p<\infty $
\begin{multline}\label{eq:pge2part}\Big(\sum_{l=0}^\infty \sum_{t\in E^\mu \cap I_l}
\|A_t^{l+n} L_{l+n} (f v_p \chi_{l+n+{{\mathfrak{z}}}})\|_p^p\Big)^{1/p} \\ {\lesssim} \min\{ 2^{-{{\mathfrak{z}}} \frac{\mu+2}{(\mu+1)p}} , 2^{{{\mathfrak{z}}} \frac{U(p)}{p} }\} 2^{-n \frac{U(p)}{p}}
\Big(\sum_{l=0}^\infty \|f\chi_{l+n+{{\mathfrak{z}}}}\|_p^p \Big)^{1/p} .\end{multline}

Of course, $(\sum_{l=0}^\infty \|f\chi_{l+n+{{\mathfrak{z}}}}\|_p^p )^{1/p} {\lesssim} \|f\|_p.$
For $p(\mu)$ we have $U(p)>0$. Hence, for $p(\mu)<p<\infty$ we may sum in ${{\mathfrak{z}}}, n$ above and obtain the desired result \eqref{eq:Iaest} from \eqref{eq:ple2part} when $p(\mu)<p\le 2$ and \eqref{eq:pge2part} when $2\le p<\infty$.
\end{proof}
We now turn to the term \eqref{eq:IIa} and argue similarly.
\begin{lemma}\label{lem:IIaest}
For $p(\mu)<p<\infty $,
\begin{equation}\label{eq:IIaest}\Big\|
\sup_l \sup_{t\in E^\mu\cap I_l} \Big| \sum_{l\frac{\mu}{\mu+1}\le j\le l} A_t^j L_j(f v_p ) \Big|
\Big\|_{L^p(|x|\sim 1)} {\lesssim} \|f\|_p
\end{equation}
\end{lemma}
\begin{proof}
We now estimate
\begin{align*} &\Big\|
\sup_l \sup_{t\in E^\mu\cap I_l} \Big| \sum_{l\frac{\mu}{\mu+1} \le j\le l} A_t^j L_j(f v_p ) \Big|
\Big\|_{L^p(|x|\sim 1)}
\\&=
\Big\|
\sup_l \sup_{t\in E^\mu\cap I_l} \Big| \sum_{-\frac{l}{\mu+1} \le n\le 0} A_t^{l+n} L_j(f v_p ) \Big|
\Big\|_{L^p(|x|\sim 1)}
\\&{\lesssim} \sum_{n=-\infty}^{-1} \sum_{{{\mathfrak{z}}}\in {{\mathbb{Z}}}}
\Big\|\sup_{l\ge -n(\mu+1)} \big|
\sup_{t\in E^\mu\cap I_l} A_t^{l+n} L_{l+n} (f v_p \chi_{l+n+{{\mathfrak{z}}}}) \big|\Big\|_{L^p(|x|\sim 1)}
\\
&{\lesssim} \sum_{n=-\infty}^{-1} \sum_{{{\mathfrak{z}}}\in {{\mathbb{Z}}}}
\Big(\sum_{l\ge -n(\mu+1)} \sum_{\substack{s\in E^\mu \cap I_l\\ 2^{-l-n}\text{-separated}}}
\Big\|\sup_{|t-s|\le 2^{-l-n} } |A_t^{l+n} L_{l+n} (f v_p \chi_{l+n+{{\mathfrak{z}}}})|\Big\|_p^p\Big)^{1/p}
\end{align*}
Now an argument using the fundamental theorem of calculus gives
\begin{align*}
&\|\sup_{|t-s|\le 2^{-l-n} } |A_t^{l+n} L_{l+n} (f v_p \chi_{l+n+{{\mathfrak{z}}}})|\|_p
\\&{\lesssim} 2^{-(l+n)(d-1)/p'} \| L_{l+n} (f v_p \chi_{l+n+{{\mathfrak{z}}}})\|_p \\ &{\lesssim} 2^{-(l+n)(d-1)/p'}2^{(l+n+{{\mathfrak{z}}}) \frac{U(p)}{p}} \min\{ 1, 2^{-{{\mathfrak{z}}} d/p'} \}\|f\chi_{l+n+{{\mathfrak{z}}}}\|_p,
\end{align*}
by Lemma \ref{lem:Gamma}. For $n\le 0$, and $l\ge -n(\mu+1)$ a maximal $2^{-l-n}$ separated set of $I_l$ has cardinality ${\lesssim} 2^{\frac{l}{\mu+1}+n}$.
Hence, for $1\le p\le 2$,
\begin{align*}
&\Big(\sum_{l\ge -n(\mu+1)} \sum_{\substack{s\in E^\mu \cap I_l\\ 2^{-l-n}\text{-separated}}}
\Big\|\sup_{|t-s|\le 2^{-l-n} } |A_t^{l+n} L_{l+n} (f v_p \chi_{l+n+{{\mathfrak{z}}}})|\Big\|_p^p\Big)^{1/p}
\\&{\lesssim}
\Big(\sum_{l\ge -n(\mu+1)} 2^{\frac{l}{\mu+1}+n}
2^{-(l+n)(d-1)(p-1) }2^{(l+n+{{\mathfrak{z}}})U(p)}
\min\{1, 2^{-{{\mathfrak{z}}} d(p-1)}\} \|f\chi_{l+n+{{\mathfrak{z}}}}\|_p^p\Big)^{1/p}
\\&{\lesssim}\Big(\sum_{l\ge -n(\mu+1)} 2^{\frac{n\mu}{\mu+1} } 2^{{{\mathfrak{z}}} U(p) }\min\{1, 2^{-{{\mathfrak{z}}} d(p-1)}\} \| f \chi_{l+n+{{\mathfrak{z}}}}\|_p^p\Big)^{1/p}
\\
&{\lesssim} 2^{\frac{n\mu}{(\mu+1)p} }\min\{ 2^{\frac{{{\mathfrak{z}}} U(p)}{p}} , 2^{-{{\mathfrak{z}}}(\frac{1}{p'} +\frac{1}{(\mu+1)p})} \} \|f\|_p
\end{align*}
Again by an interpolation with an $L^\infty$ estimate we get for $2\le p<\infty$, $\alpha=U(p)$

\begin{align*}
&\Big(\sum_{l\ge -n(\mu+1)} \sum_{\substack{s\in E^\mu \cap I_l\\ 2^{-l-n}\text{-separated}}}
\Big\|\sup_{|t-s|\le 2^{-l-n} } |A_t^{l+n} L_{l+n} (f v_p \chi_{l+n+{{\mathfrak{z}}}})|\Big\|_p^p\Big)^{1/p}
\\&{\lesssim} 2^{\frac{n\mu}{(\mu+1)p} }\min\{ 2^{\frac{{{\mathfrak{z}}} U(p)}{p}} , 2^{-{{\mathfrak{z}}}(\frac{\mu+2}{(\mu+1)p})} \} \|f\|_p
\end{align*}
For $p>p(\mu)$ we may sum in $n\le 0$ and ${{\mathfrak{z}}}$ to obtain \eqref{eq:IIaest}
\end{proof}


\bigskip

\begin{thebibliography}{10}

\bibitem{BakOberlinSeegerAJM}
Jong-Guk Bak, Daniel~M. Oberlin, and Andreas Seeger.
\newblock Restriction of {F}ourier transforms to curves and related oscillatory
integrals.
\newblock{\em Amer. J. Math.}, 131(2):277--311, 2009.

\bibitem{BCSS26}
Riju Basak, Surjeet~Singh Choudhary, Chun-Yen Shen, and Daniel Spector.
\newblock Sharp power-weight estimates for the fractional spherical maximal
operator.
\newblock arXiv:2610.13326.

\bibitem{BeltranRoosRutarSeeger}
David Beltran, Joris Roos, Alex Rutar, and Andreas Seeger.
\newblock A fractal local smoothing problem for the wave equation.
\newblock{\em Bull. Lond. Math. Soc.}, 57(12):3667--3690, 2025.

\bibitem{BeltranRoosSeeger-radial}
David Beltran, Joris Roos, and Andreas Seeger.
\newblock Spherical maximal operators with fractal sets of dilations on radial
functions.
\newblock{\em Studia Math.}, 289(1):1--32, 2026.

\bibitem{bergh-lofstrom}
J\"{o}ran Bergh and J\"{o}rgen L\"{o}fstr\"{o}m.
\newblock{\em Interpolation spaces. {A}n introduction}.
\newblock Grundlehren der Mathematischen Wissenschaften, No. 223.
Springer-Verlag, Berlin-New York, 1976.

\bibitem{BourgainJdA1986}
Jean Bourgain.
\newblock Averages in the plane over convex curves and maximal operators.
\newblock{\em J. Analyse Math.}, 47:69--85, 1986.

\bibitem{CarberyConvexBody}
Anthony Carbery.
\newblock An almost-orthogonality principle with applications to maximal
functions associated to convex bodies.
\newblock{\em Bull. Amer. Math. Soc. (N.S.)}, 14(2):269--273, 1986.

\bibitem{Christ-frac}
Michael Christ.
\newblock Endpoint bounds for singular fractional integral operators.
\newblock Unpublished manuscript (1988).

\bibitem{DuoandikoetxeaSeijo}
Javier Duoandikoetxea and Edurne Seijo.
\newblock Weighted inequalities for some spherical maximal operators.
\newblock{\em Illinois J. Math.}, 46(4):1299--1312, 2002.

\bibitem{DuoandikoetxeaVega}
Javier Duoandikoetxea and Luis Vega.
\newblock Spherical means and weighted inequalities.
\newblock{\em J. London Math. Soc. (2)}, 53(2):343--353, 1996.

\bibitem{FraccaroliRoosSeeger}
Marco Fraccaroli, Joris Roos, and Andreas Seeger.
\newblock Power weight inequalities for spherical maximal functions.
\newblock{\em Math. Z.}, 314(1):Paper No. 4, 2026.

\bibitem{FraserHareHareTroscheitYu}
Jonathan~M. Fraser, Kathryn~E. Hare, Kevin~G. Hare, Sascha Troscheit, and Han
Yu.
\newblock The {A}ssouad spectrum and the quasi-{A}ssouad dimension: a tale of
two spectra.
\newblock{\em Ann. Acad. Sci. Fenn. Math.}, 44(1):379--387, 2019.

\bibitem{GarciaCuerva-Rubio}
Jos\'{e} Garc\'{\i}a-Cuerva and Jos\'{e}~L. Rubio~de Francia.
\newblock{\em Weighted norm inequalities and related topics}, volume 116 of
{\em North-Holland Mathematics Studies}.
\newblock North-Holland Publishing Co., Amsterdam, 1985.
\newblock Notas de Matem\'{a}tica [Mathematical Notes], 104.

\bibitem{GreenleafSeegerWainger}
Allan Greenleaf, Andreas Seeger, and Stephen Wainger.
\newblock On {X}-ray transforms for rigid line complexes and integrals over
curves in {${\bf R}^4$}.
\newblock{\em Proc. Amer. Math. Soc.}, 127(12):3533--3545, 1999.

\bibitem{Kurtz80}
Douglas~S. Kurtz.
\newblock Littlewood-{P}aley and multiplier theorems on weighted {$L\sp{p}$}\
spaces.
\newblock{\em Trans. Amer. Math. Soc.}, 259(1):235--254, 1980.

\bibitem{LabaPramanik-Duke}
Izabella {\L}aba and Malabika Pramanik.
\newblock Maximal operators and differentiation theorems for sparse sets.
\newblock{\em Duke Math. J.}, 158(3):347--411, 2011.

\bibitem{LeeJ-weighted}
Juyoung Lee.
\newblock The critical weighted inequalities of the spherical maximal function.
\newblock{\em J. Geom. Anal.}, 35(8):Paper No. 220, 17, 2025.

\bibitem{NagelSteinWainger-lac}
Alexander Nagel, Elias~M. Stein, and Stephen Wainger.
\newblock Differentiation in lacunary directions.
\newblock{\em Proc. Nat. Acad. Sci. U.S.A.}, 75(3):1060--1062, 1978.

\bibitem{RoosSeeger-problems}
Joris Roos and Andreas Seeger.
\newblock Problems on spherical maximal functions.
\newblock To appear in Bull. Amer. Math. Soc., arXiv:2511.11283, 2025.

\bibitem{SeegerTao}
Andreas Seeger and Terence Tao.
\newblock Sharp {L}orentz space estimates for rough operators.
\newblock{\em Math. Ann.}, 320(2):381--415, 2001.

\bibitem{STW-jussieu}
Andreas Seeger, Terence Tao, and James Wright.
\newblock Endpoint mapping properties of spherical maximal operators.
\newblock{\em J. Inst. Math. Jussieu}, 2(1):109--144, 2003.

\bibitem{SeegerWainger-frac}
Andreas Seeger and Stephen Wainger.
\newblock Bounds for singular fractional integrals and related {F}ourier
integral operators.
\newblock{\em J. Funct. Anal.}, 199(1):48--91, 2003.

\bibitem{SeegerWaingerWright1995}
Andreas Seeger, Stephen Wainger, and James Wright.
\newblock Pointwise convergence of spherical means.
\newblock{\em Math. Proc. Cambridge Philos. Soc.}, 118(1):115--124, 1995.

\bibitem{SeegerWaingerWright1997}
Andreas Seeger, Stephen Wainger, and James Wright.
\newblock Spherical maximal operators on radial functions.
\newblock{\em Math. Nachr.}, 187:241--265, 1997.

\bibitem{SteinPNAS1976}
Elias~M. Stein.
\newblock Maximal functions. {I}. {S}pherical means.
\newblock{\em Proc. Nat. Acad. Sci. U.S.A.}, 73(7):2174--2175, 1976.

\end{thebibliography}
\end{document}